\documentclass[11pt]{article}

\usepackage[margin=1.3in]{geometry}
\usepackage{enumitem}



\newcommand{\cX}{\mathsf{X}}

\newcommand{\R}{\mathbb{R}}

\newcommand{\bP}{\mathbf{P}}

\newcommand{\rint}{\operatorname{int}}

\newcommand{\N}{\mathbb{N}}
\newcommand{\diag}{\operatorname{diag}}

\usepackage{dsfont}
\usepackage{amssymb}
\usepackage{amsmath}
\usepackage{url}
\usepackage{xcolor}
\usepackage{amsthm}
\usepackage{caption}
\usepackage{subcaption}
\usepackage{graphicx}

\usepackage{tikz}
\usetikzlibrary{shapes,arrows,graphs,graphs.standard,shapes.misc}
\usetikzlibrary{matrix,decorations.pathreplacing, calc, positioning,fit}
\usetikzlibrary{shapes.geometric}
\usetikzlibrary{decorations.markings}

\usepackage{pgfplots}
\pgfplotsset{compat=1.8}
\usepackage{pgfplotstable}
\usetikzlibrary{3d, calc, decorations.markings}
\usepackage{placeins}
\pgfplotsset{compat=1.17}

\usepackage[
    style=numeric-comp,
    hyperref=true,
    doi=false,
    url=false,
    isbn=false,
    firstinits=true,
    sorting=none,
    block=none,
    backend=bibtex,
    maxnames=99
]{biblatex}

\bibliography{references}

\newtheorem{theorem}{Theorem}
\newtheorem{lemma}{Lemma}
\newtheorem{proposition}{Proposition}
\newtheorem{definition}{Definition}
\newtheorem{remark}{Remark}

\title{Convergence rate of $H$-property for step-graphons}

\author{Wanting Gao \quad Hong Hu \quad and \quad Xudong Chen
\thanks{The authors are with the Electrical and Systems Engineering, Washington University in St. Louis. Emails: \texttt{\{g.wanting, hu.h, cxudong\}@wustl.edu}.}%
}

\date{}

\begin{document}

\maketitle

\begin{abstract}
A graphon is said to have the $H$-property if a random undirected graph $G_n$ on $n$ nodes sampled from it has a node-wise disjoint cycle cover almost surely as $n\to\infty$. 
It has been shown in the earlier work that the $H$-property obeys the zero-one law, i.e., the probability that the random graph has a cycle cover tends to either one or zero.  
In this paper, we sharpen the result by characterizing the convergence rate of the probability. Specifically, we show that there are two different types of rates, with one being exponential and the other being root $n$. We provide a rigorous proof and numerical validation.
\end{abstract}

\section{Introduction and Main Result}

In this paper, we investigate the convergence rate of $H$-property for step-graphons. Roughly speaking, a step-graphon $W$ has the (weak) $H$-property if a random graph $G_n$ on $n$ nodes sampled from $W$ has a node-wise disjoint cycle cover (i.e., a node-wise disjoint union of cycles that visit every node of the graph) almost surely as $n\to \infty$.   In~\cite{belabbas2021h,belabbas2023geometric}, it is shown that the $H$-property obeys the zero-one law, i.e., for almost all step-graphons, the probability $P_n(W)$ that $G_n\sim W$ has a cycle cover converges to either zero or one. A complete characterization of the zero-one law has also been obtained, for which we have introduced key objects such as concentration vector, skeleton graph, and edge cone and used these objects to establish necessary and sufficient conditions for the $H$-property. 
The residual case where the probability converges to neither zero nor one has also been addressed in~\cite{gao2025h}. We recall these key objects and the zero-one law in Subsection~\ref{ssec:Hproperty}. 

The importance of cycle cover for networked control system lies in the fact that it is necessary and essentially sufficient (together with some other mild conditions) for a network topology to sustain asymptotic stability~\cite{kirkoryan2014decentralized} and ensemble controllability~\cite{chen2021sparse}. 
Understanding how likely a random network topology can admit a cycle cover is key to evaluating the risk-to-reward ratio for a network manager to deploy a large-scale multi-agent system in an uncertain environment.  We refer the reader to the tutorial paper~\cite{belabbas2025structural} for details.    

The main contribution of this paper is to characterize the convergence rate of $P_n(W)$. Specifically, we establish the following results, all of which are formulated in Theorem~\ref{thm:convergencerate}, Subsection~\ref{ssec:mainresult}:  
\begin{enumerate}
\item If $P_n(W)$ converges to $1$, then it converges exponentially fast.
\item If $P_n(W)$ converges to $0$, then there are two subcases:
\begin{enumerate}
\item If the concentration vector does not belong to the edge cone, then $P_n(W)$ converges to $0$ exponentially fast.
\item If the concentration vector is on the boundary of the edge cone and if the edge cone is degenerate, then $P_n(W) \sim O(n^{-1/2})$.
\end{enumerate}
\item If $P_n(W)$ converges to $p^*$ for some $p^*\in (0,1)$, then $|P_n(W) - p^*| \sim O(n^{-1/2})$.
\end{enumerate}

\vspace{.1cm}

\noindent\textit{Notation.} We gather here a few notations that are used throughout the paper. 
For an undirected graph $G$, let $V(G)$ and $E(G)$ be its node set and edge set. We use $(v_i,v_j)$ to denote the edge between $v_i$ and $v_j$. If $v_i = v_j$, then $(v_i,v_i)$ is a self-loop. 
For $x = (x_1, \dots, x_n) \in \mathbb{R}^n$, let diag($x$) be the $n \times n$ diagonal matrix whose $ii$-th entry is $x_i$. We use $e_1,\ldots, e_n$ to denote the standard basis of $\R^n$, where  $e_i$ is the vector with $1$ on the $i$-th entry and $0$ elsewhere.
We denote by $\Delta^{n-1}$ the standard $(n-1)$-simplex in $\R^n$, i.e., the set of all vectors $x\in \R^{n}$ with nonnegative entries and $\sum_{i = 1}^{n}x_i = 1$. For a closed convex set $A \subseteq \mathbb{R}^q$, we denote by $\operatorname{int}(A)$ its relative interior.
  
\subsection{Step-graphon, $H$-property, and key objects}\label{ssec:Hproperty}
A graphon $W:[0,1]^2\to [0,1]$ is a symmetric, measurable function. 
We follow the steps below to sample an undirected graph $G_n$ on $n$ nodes from $W$: 
\begin{enumerate}
    \item Sample $r_1, \ldots, r_n$ i.i.d. from $\text{Uni}[0,1]$, the uniform distribution over the interval $[0,1]$. We call $r_i$ the \emph{coordinate} of node $v_i \in V(G_n)$.
    \item For any two {\em distinct} nodes $v_i$ and $v_j$, place an edge $(v_i, v_j)$ with probability $W(r_i, r_j)$.
\end{enumerate}

Note that $G_n$ is simple, i.e., it does not have self-loop nor multiple edges. 
Given a step-graphon $W$, let 
\begin{equation}\label{eq:defPn}
P_n(W):= \mathbf{P}(G_n \sim W \mbox{ has a cycle cover}).
\end{equation} 
We have the following definition for $H$-property: 

\begin{definition}[$H$-property]\label{def:Hproperty} 
A graphon $W$ has the {\bf $H$-property} if
$$
\lim_{n \to \infty} P_n(W) = 1.
$$
\end{definition}

In this paper, we focus on the class of step-graphons: 

\begin{definition}[Step-graphon and its partition]\label{def:stepgraphon}
A graphon $W$ is a {\bf step-graphon} if there exists a sequence $0 = \sigma_0 < \sigma_1 < \cdots < \sigma_q = 1$ such that $W$ is constant on each rectangle $[\sigma_i, \sigma_{i+1}) \times [\sigma_j, \sigma_{j+1})$ for all $0 \leq i,j \leq q-1$. We call 
$\sigma = (\sigma_0, \sigma_1, \ldots, \sigma_q)$ a {\bf partition} for $W$.
\end{definition}

We now introduce the key objects that decide the zero-one law for the $H$-property: 

\begin{definition}
[Concentration vector]\label{def:concentrationv} Let $W$ be a step-graphon with partition $\sigma = (\sigma_0, \ldots, \sigma_q)$. The associated {\bf concentration vector} $x^* = (x^*_1, \ldots, x^*_q)$ is such that
$x^*_i := \sigma_i - \sigma_{i-1}$, for all $i = 1, \ldots, q$.
\end{definition}

Note that $x^*$ is a probability vector with all entries positive.  

We next have the following definition:

\begin{definition}
[Skeleton graph]\label{def:SkeletonGraph} To a step-graphon $W$ with a partition $\sigma = (\sigma_0, \ldots, \sigma_q)$, we assign the undirected graph $S$ on $q$ nodes, with $V(S) = \{u_1, \ldots, u_q\}$ and edge set $E(S)$ defined as follows: there is an edge between $u_i$ and $u_j$ if and only if $W$ is nonzero over $[\sigma_{i-1}, \sigma_i) \times [\sigma_{j-1}, \sigma_j)$. We call $S$ the {\bf skeleton graph} of $W$ for the partition $\sigma$. 
\end{definition}

Note that the concentration vector $x^*$ and the skeleton graph $S$ together completely determine the support of the graphon $W$, and vice versa. 

For each edge $f_j = (u_{i_1}, u_{i_2})$ of $S$, we let $z_j$ be the corresponding incidence vector: 
$$z_j := \frac{1}{2} (e_{i_1} + e_{i_2})  \in \R^q.$$
Let $Z := [z_1,\ldots, z_k]$ be the incidence matrix. We have

\begin{definition}[Edge cone]\label{def:edgecone} 
The {\bf edge cone} $\cX$ of $S$ is the convex cone generated by the  vectors $z_1,\ldots, z_k$, i.e.,
$$
    \cX := \left \{\sum_{j = 1}^k c_j z_j \mid c_j \geq 0 \mbox{ for all } j = 1,\ldots, k\right \}.
$$

\end{definition}

\begin{figure}[t]
\centering

\begin{minipage}[t]{0.46\columnwidth}
\vspace{0pt}
    \centering
    
    \begin{tikzpicture}[scale=.50]
    \fill[bottom color=white,top color=black] (-1.8,0) rectangle (-2.6,4) node [left] {\small$1$};
\node [left] at (-2.6,0) {\small$0$};

        \filldraw [fill=gray!70!black!40, draw=gray!70!black!40] (1.2,0) rectangle (2.4,1.6);
        \filldraw [fill=gray!70!black!40, draw=gray!70!black!40] (2.4,1.6) rectangle (4,2.8);
        \filldraw [fill=black, draw=black] (2.4,0) rectangle (4,1.6);
        \filldraw [fill=gray!60!black!70, draw=gray!60!black!70] (0,1.6) rectangle (1.2,2.8);
        \filldraw [fill=gray!60!black!70, draw=gray!60!black!70] (1.2,2.8) rectangle (2.4,4);
        \filldraw [fill=black, draw=black] (0,2.8) rectangle (1.2,4);

        \draw [draw=black,very thick] (0,0) rectangle (4,4);

        \node [above] at (0,4) {$0$};
        \node [above] at (1.2,4) {$.3$};
        \node [above] at (2.4,4) {$.6$};
        \node [above] at (4,4) {$1$};
    \end{tikzpicture}
    
    \vspace{3mm}
    
    \begin{tikzpicture}[scale=0.85]
        \tikzset{every loop/.style={}}
        \node at (-0.7,0) {$S$};
        \node [circle,fill=black,inner sep=1.2pt,label=below:{$u_1$}] (1) at (0, 0) {};
        \node [circle,fill=black,inner sep=1.2pt,label=below:{$u_2$}] (2) at (.8, 0) {};
        \node [circle,fill=black,inner sep=1.2pt,label=below:{$u_3$}] (3) at (1.6, 0) {};
        
        \path[draw,thick,shorten >=2pt,shorten <=2pt]
            (1) edge[loop left] (1)
            (3) edge[loop right] (3)
            (1) edge (2)
            (2) edge (3);
    \end{tikzpicture}
\end{minipage}
\begin{minipage}[t]{0.50\columnwidth}
\vspace{0pt}
    \centering
    
    \begin{tikzpicture}[scale=1.15]
        \tikzset{every loop/.style={}}
        \node [circle,fill=black,inner sep=0.8pt] (0) at (0,0) {};
        \draw[->] (0,0) -- (1.55,0) node [above] {\footnotesize $x_1$};
        \draw[->] (0,0) -- (0,1.55) node [right] {\footnotesize $x_2$};
        \draw[->] (0,0) -- (-0.9,-0.9) node [left] {\footnotesize $x_3$};

        \node (1) at (1.2, 0) {};
        \node (2) at (0, 1.2) {};
        \node (3) at (-0.66, -0.66) {};
        \node (12) at (0.6, 0.6) {};
        \node (23) at (-0.33, 0.27) {};
        
        \filldraw[draw=black, fill=white]
            (1.center) -- (2.center) -- (3.center) -- cycle;

        \filldraw[draw=black, fill=blue!20]
            (1.center) -- (12.center) -- (23.center) -- (3.center) -- cycle;

        \node (xsl) at (1.0,1.0) {\small \textcolor{blue}{$\cX(S) \cap \Delta^2$}};
        \node (xst) at (0.7, 0.1) {};

        \path[draw,shorten >=3pt,shorten <=3pt]
            (xsl) edge[-latex] (xst);

        \node [circle,fill=red,inner sep=1.0pt,
        label=below:{\color{red} $x^*$}] at (0.096,0.096) {};
    \end{tikzpicture}
\end{minipage}

\caption{{\it Left:} A step-graphon $W$ and its skeleton graph $S$. {\it Right:} The associated edge cone $\cX(S)$ and concentration vector $x^*$ plotted in the standard $2$-simplex.}
\label{fig:graphon-skeleton-simplex}
\end{figure}
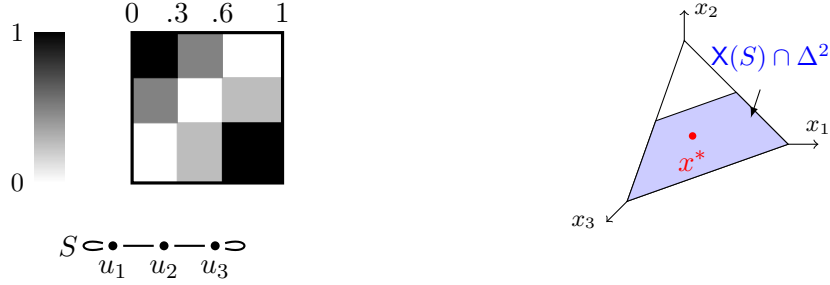

Fig.~\ref{fig:graphon-skeleton-simplex} provides a geometric illustration of the three key objects associated with graphon $W$ introduced above: the skeleton graph $S$, the edge cone $\cX(S)$ embedded in the standard simplx, together with the concentration vector $x^*$.

The following result~\cite{ohsugi1998normal} relates the dimension of $\cX$ to the existence of an odd cycle in $S$. A cycle is said to be \emph{odd} if its length is odd (a self-loop is an odd cycle of length $1$):

\begin{lemma}\label{lem:fullrankness}
Suppose that $S$ is connected; then, 
\[
\dim \, \cX = 
\begin{cases} 
q  & \text{if } S \text{ has an odd cycle,} \\
q - 1 & \text{otherwise.}
\end{cases}
\]
\end{lemma}

\subsection{Main result}\label{ssec:mainresult}

Given a step-graphon $W$, let $x^*$ and $S$ be the associated concentration vector and the skeleton graph, respectively. 
We assume in the sequel that $S$ is connected, which is equivalent to the condition that $W$ is not block-diagonal. If $S$  has multiple connected components, then our results state below can be applied to each connected component of $S$. 
We introduce the following three conditions: 
\begin{description}[leftmargin=2em]
\item[\it $A$:] $S$ has an odd cycle. 
\item[\it $B$:] $x^*\in \rint \cX$.  
\item[\it $B'$:] $x^* \in \cX$. 
\end{description}
The main theorem of the paper is stated below: 

\begin{theorem}[Main Theorem]\label{thm:convergencerate}
Let $W$ be a step-graphon and $P_n(W)$ be given in~\eqref{eq:defPn}. 
Then, the following items hold:
\begin{enumerate}
\item If $W$ satisfies Conditions $A$ and $B$, then there exist positive constants $c$ and $\lambda$ such that
\begin{equation}\label{eq:case1}
    1 - P_n(W) \leq c\, e^{-\lambda n}.
\end{equation}

\item If $W$ does not satisfy Condition $B'$, then there exist positive constants $c$ and $\lambda$ such that

\begin{equation}\label{eq:case2}
P_n(W) \leq c \, e^{-\lambda n}. 
\end{equation}

\item If $W$ satisfies Condition $B'$, but not $A$, then there exist a positive constant $c$ such that 
\begin{equation}\label{eq:case3}
P_n(W) \leq \frac{c}{\sqrt{n}}.
\end{equation}

\item If $W$ satisfies Condition $A$ and $B'$, but not $B$, then there exist a $p^*\in (0,1)$ and a positive constant $c$ such that
\begin{equation}\label{eq:case4}
|P_n(W) - p^*| \leq \frac{c}{\sqrt{n}}.  
\end{equation}
\end{enumerate}
\end{theorem}

\begin{remark}\normalfont
Note that the zero-one law for the $H$-property has been established  in~\cite{belabbas2021h,belabbas2023geometric,gao2025h}. The contribution of the above theorem is not about characterizing when $P_n(W)$ converges to $0$ or $1$ (or some value $p^*\in (0,1)$), but rather characterizing the convergence rate. 
The value $p^*$ in item~4 of the theorem admits an explicit expression, which we will introduce in Subsection~\ref{ssec:preliminaries}. 
\hfill{\qed}
\end{remark}

The remainder of the paper is organized as follows. In Section~\ref{sec:proof}, we prove the main theorem. Specifically, in Subsection~\ref{ssec:preliminaries}, we present the preliminary results that will be used throughout the proofs, including an explicit characterization of the limiting value $p^*$ in the residual case, necessary and sufficient conditions for the existence of a cycle cover for $S$-partite graphs (see Definition~\ref{def:Spartitegraph}), the use of 
the Blow-up Lemma, and relevant properties of multi-nomial random variables. 
Next, in Subsection~\ref{ssec:expupperbound}, we establish the exponential upper bounds stated in items~1 and~2 of the main theorem. 
Then, in Subsection~\ref{ssec:rootnupperbound}, we establish the root-$n$ convergence for items~3 and~4. 
Finally, in Section~\ref{sec:numericalstudy}, we provide numerical validation for the four items of the main theorem. The paper ends with conclusions. 

\section{Proof of the Main Theorem}\label{sec:proof}

\subsection{Preliminaries}\label{ssec:preliminaries}

\subsubsection{On the value $p^*$ for the residual case}

In this subsection, we provide an expression for the value $p^*$ in item~4 of Theorem~\ref{thm:convergencerate}. 
To this end, we first recall the \emph{facet-defining hyperplane} $H_k$ of the edge cone $\cX$. $H_k$ is a co-dimensional one subspace of $\mathbb{R}^q$ that satisfied the following two conditions:
\begin{enumerate}
    \item There exist $(q- 1)$ linearly independent vectors $z_{k_1}, \dots, z_{k_{q-1}}$ out of the columns of the matrix $Z$ such that $H_k$ is spanned by these vectors.
    \item The subspace $H_k$ is a supporting hyperplane for $\cX$, i.e., there exists a unique vector $v_k \in \mathbb{R}^q$ of unit length perpendicular to $H_k$, such that $v_k^\top x\geq0,$ for all $x \in \cX.$
\end{enumerate}
We denote by $\mathcal{H}$ the collection of all facet-defining hyperplanes of $\cX$. Since the concentration vector $x^*$ lies on the boundary $\partial\cX$, there exists at least one hyperplane $H_k \in \mathcal{H}$ such that $x^* \in H_k$. We define $\mathcal{H}(x^*)$ as the set of all such hyperplanes, that is, 
\[
\mathcal{H}(x^*) = \{H_k \in \mathcal{H} \mid v_k^\top x^* = 0\}.
\]
Let $L$ be the affine hyperplane that contains the standard simplex $\Delta^{q-1}$:
\begin{equation}\label{eq:L}
    L := \left\{ x \in \mathbb{R}^q \,\middle|\, \mathbf{1}^\top x = 1 \right\}.
\end{equation}

The associated affine convex cone is then given by
\[
\Omega^* := \{ \omega \in L \mid v_k^\top \omega \ge 0 \text{ for all } H_k \in \mathcal{H}(x^*) \}.
\]
We have the following result from~\cite{gao2025h}:

\begin{lemma}\label{lem:residualcase}
    If $x^* \in \partial \cX$ and $S$ has an odd cycle, then 
    \[
    p^* = \lim_{n\to \infty}P_n(W) = \bP(\omega^* \in  \Omega^*) .
    \] 
\end{lemma}

\subsubsection{On the use of Blow-up Lemma and almost sure $(\epsilon,\delta)$-regularity}
We recall the Blow-up Lemma from the extremal graph theory  and describe its use in this paper. 

We start with the following definition:

\begin{definition}[$S$-partite graph]\label{def:Spartitegraph}
Let $S$ be an undirected graph on $q$ nodes, possibly with self-loops.
A simple, undirected graph $G$ is {\bf $S$-partite} if there exists a graph homomorphism
$\pi :G \to S$. 
Further, $G$ is a \emph{complete $S$-partite graph} if
\[
(v_i, v_j) \in E(G)
\quad \Longleftrightarrow \quad
(\pi(v_i), \pi(v_j)) \in E(S).
\]
\end{definition}

Given a vector $y\in \N^q$, we let $K_y$ be the complete $S$-partite graph such that $y_i:= |\pi^{-1} (u_i)|$ for all $u_i\in V(S)$, i.e., $y_i$ is the size of the $i$-th community. We first have the following result~\cite{jiang2026largest}:  

\begin{lemma}\label{lem:hamiltonicityofKx}
If $y\in \cX$ is integer-valued and if $y_i \geq 3$ for all $i = 1,\ldots, q$, then $K_y$ has a cycle cover. \end{lemma}

Conversely, we have

\begin{lemma}\label{lem:edgecone}
    If an $S$-partite graph $G$ has a cycle cover, then $x(G_n)\in \cX$. 
\end{lemma}

Next we recall the Blow-up Lemma~\cite{komlos1997blow}, which states that if a subgraph $G$ of $K_y$, with $V(G) = V(K_y)$, satisfies some regularity condition, then any graph of bounded degree, which is embeddable into $K_y$, is also embeddable into $G$. A precise description of the Blow-up Lemma will be given later in this subsection.

For two disjoint subsets $X$ and $Y$ of $V(G)$,
let $e(X,Y)$ be the number of edges between $X$ and $Y$. We need the following definition:
\begin{definition}[Super-regular pair]
Let $G$ be an undirected graph, and $A,B$ be two disjoint subsets of $V(G)$.
The pair $(A,B)$ is {\bf $(\epsilon,\delta)$-super-regular} if
\begin{equation}
\begin{aligned}
e(X,Y) &> \delta |X||Y|, 
\\
&\text{for any } X \subseteq A,\; Y \subseteq B
\\
&\text{with } |X| > \epsilon |A|,\; |Y| > \epsilon |B|,
\end{aligned}
\end{equation}
and, moreover,
\begin{equation}
\begin{aligned}
e(a,B) &> \delta |B|, && \mbox{for all } a \in A, \\
e(b,A) &> \delta |A|, && \mbox{for all } b \in B.
\end{aligned}
\end{equation}
\end{definition}

We extend the above definition to the $S$-partite graphs.

\begin{definition}[Super-regular $S$-partite graphs]
An $S$-partite graph $G$ is
{\bf $(\epsilon,\delta)$-super-regular} if for any two distinct nodes
$u_i,u_j \in V(S)$, $(\pi^{-1}(u_i),\pi^{-1}(u_j))$ is
 $(\epsilon,\delta)$-super-regular.
\end{definition}

The following result is a direct consequence of 
the Blow-up Lemma, specialized to the case where the subgraph $H$ is a cycle cover (and hence, the degree of $H$ is $2$).  

\begin{lemma}[Blow-up Lemma~\cite{komlos1997blow}]\label{lem:blowup}
Let $S$ be an undirected graph on $q$ nodes, without self-loops.
Given any parameter $\delta > 0$,
there exists an $\epsilon = \epsilon(\delta,q) > 0$
such that the following holds for any $y\in \N^q$:
If $K_y$ has a cycle cover, then any
$(\epsilon,\delta)$-super-regular $S$-partite graph $G$, with
$y(G)=y$, has a cycle cover as well.
\end{lemma}

Next, let $p_{ij}$ be the value of $W$ over the rectangle $[\sigma_{i-1},\sigma_i)\times [\sigma_{j-1},\sigma_j)$ and 
\begin{equation}\label{eq:delta}
\delta:= \min \{p_{ij} \mid (u_i,u_j)\in E(S)\}.
\end{equation}
For convenience, we let 
\begin{equation}
    \mathcal{E}_n := \mbox{the event that } G_n\sim W \mbox{ is $(\epsilon,\delta)$-super-regular}.\label{eq:superregular}
\end{equation}

We state the following fact: 

\begin{lemma}\label{lem:expconvregular}
Let $\delta$ be given as in~\eqref{eq:delta}. Then, for any $\epsilon>0$ and for any closed set $\mathcal{B}$ in $\rint \Delta^{q-1}$, there exist positive constants $c$ and $\lambda$ such that
$$
\bP(\neg\mathcal{E}_n \mid x(G_n)\in \mathcal{B}) \\ \leq c \, e^{-\lambda n}. 
$$
\end{lemma}

We omit the proof of the lemma, which can be established by standard arguments in random graph theory. We also refer the reader to~\cite{chen2025hamiltonicity} for arguments of a similar result (specifically, Proposition~8 and its proof in Appendix~C). 

\subsubsection{On multinomial random variable}

Given $G_n\sim W$, we let $y_i(G_n):= |\pi^{-1}(u_i)|$ for all $u_i\in V(S)$ and 
$y(G_n) := (y_1(G_n),\ldots, y_q(G_n))$. 
It is clear from the sampling procedure (step 1) that $y(G_n)$ is a multinomial random variable with $n$ trials, $q$ events, and $x^*_i$ the probability of the $i$-th event. Further, let $$x(G_n):= \frac{1}{n} y(G_n).$$
We call $x(G_n)$ the {\it empirical concentration vector}. By law of large numbers, $x(G_n)$ converges to its mean, i.e., the concentration vector $x^*$ as $n \to \infty$.

As we will see in the following subsections, the convergence rate of $P_n(W)$ is essentially determined by the rate of  $\bP(x(G_n)\in \cX)$. For this reason, we gather in this subsection a few relevant results about the random variable $x(G_n)$.  We start with the following lemma: 

\begin{lemma}\label{lem:upexpcvg}
For any $\delta > 0$, 
$$
\bP(\|x(G_n) - x^*\|_2 \leq \delta) \geq 1 - 2q \exp(-2n\delta^2/q).
$$
\end{lemma}
The above lemma shows that the empirical concentration vector $x(G_n)$ concentrates around its mean $x^*$ at an exponential rate. This result follows directly from applying Hoeffding’s inequality~\cite{hoeffding1963probability} coordinate-wise and a union bound. Details are omitted.

Now, let $\omega(G_n):= \sqrt{n}(x(G_n) - x^*) + x^*$ and $\omega^*$ be a Gaussian random variable with mean $x^*$ and covariance matrix $\Sigma = \diag(x^*) - x^* x^{*\top}$. It is known~\cite{arenbaev1977asymptotic} that
\begin{equation}\label{eq:omegagntostar}
\omega(G_n) \xrightarrow{d} \omega^*.
\end{equation}

Note that $\operatorname{rank}(\Sigma) = q-1$, so the support of $\omega^*$ is $L$, which is defined in~\eqref{eq:L}. We need the following result: 
 
\begin{lemma}\label{lem:berryessen}
There exists a constant $\gamma > 0$ such that for any convex set $Q$ of $L$,  
$$
|\bP(\omega(G_n)\in Q) - \bP(\omega^*\in Q)| \leq \frac{\gamma}{\sqrt{n}}.
$$
\end{lemma}
\begin{proof}
Choose $U\in\mathbb R^{q\times (q-1)}$ with full column rank such that $\Sigma=UU^\top$. Then the affine map $\Phi:\mathbb R^{q-1}\to L$ defined by $\Phi(z)=x^*+Uz$ is a bijection. Let $Y_1,\dots,Y_n\in\{e_1,\dots,e_q\}$ be i.i.d.\ with $\bP(Y_t=e_i)=x_i^*$. Then $x(G_n)=\frac1n\sum_{t=1}^n Y_t$, $\mathbb E[Y_t]=x^*$, and $\operatorname{Cov}(Y_t)=\Sigma$. Since $Y_t-x^*\in\operatorname{Im}(\Sigma)=\operatorname{Im}(U)$, there exist random vectors $Z_t\in\mathbb R^{q-1}$ such that $Y_t-x^*=UZ_t$. One checks that $\mathbb E[Z_t]=0$ and $\operatorname{Cov}(Z_t)=I_{q-1}$. Define $T_n:=\frac1{\sqrt n}\sum_{t=1}^n Z_t$. Then $\omega(G_n)=x^*+UT_n=\Phi(T_n).$
Let $G\sim N(0,I_{q-1})$. Since $x^*+UG\sim N(x^*,\Sigma)$, we have $\omega^*=\Phi(G)$. For any convex set $Q\subset L$, let $A:=\Phi^{-1}(Q)\subset\mathbb R^{q-1}$. As $\Phi$ is affine, $A$ is convex. So,
$\bP(\omega(G_n)\in Q)=\bP(T_n\in A)$ and $\bP(\omega^*\in Q)=\bP(G\in A)$. Hence
\[
|\bP(\omega(G_n)\in Q)-\bP(\omega^*\in Q)|
=
|\bP(T_n\in A)-\bP(G\in A)|.
\]
Since $Z_t$ takes finitely many values, so the third moment $\mathbb{E}\|Z_t\|^3 < \infty$. Applying the multivariate Berry--Esseen bound for convex sets ~\cite{bentkus2003dependence} yields
\[
|\bP(T_n\in A)-\bP(G\in A)|
\le
400\,(q-1)^{1/4}\frac{\mathbb E\|Z_t\|^3}{\sqrt n}.
\]
Thus the result holds with $\gamma:=400\,(q-1)^{1/4}\mathbb E\|Z_t\|^3$.
\end{proof}

These results together form the core probabilistic toolkit for analyzing the convergence of $x(G_n)$, and are key to studying the convergence of the $H$-property.

\subsection{On the upper bounds of exponential convergence}\label{ssec:expupperbound}

In this subsection, we establish the upper bounds in~\eqref{eq:case1} and~\eqref{eq:case2}. The latter is more or less straightforward, which follows directly from Lemmas~\ref{lem:upexpcvg} and~\ref{lem:edgecone}.

\begin{proposition}\label{prop:case1_2upperbound}
The following two items hold:
\begin{enumerate}
    \item If $W$ does not satisfy Condition $B'$, then there exist positive constants $c$ and $\lambda$ such that
    $$
    P_n(W) \leq c \, e^{-\lambda n}. 
    $$
    \item If $W$ satisfies Conditions $A$ and $B$, then there exist positive constants $c$ and $\lambda$ such that
    $$
    1 - P_n(W) \leq c \, e^{-\lambda n}.
    $$
\end{enumerate}

\end{proposition}

\begin{proof}
We first establish item~1 of the proposition. Let $\tau := \inf_{y\in \cX} \|x^*-y\|_2.$
\begin{align}
P_n(W)
&\le \bP\left(x(G_n)\in \cX\right)\label{eq:case2upperbound1} \\
&\le \bP\left(\|x(G_n)-x^*\|_2\ge \tau\right)\notag 
\\
&\le 2q\exp\!\left(-2n\frac{\tau^2}{q}\right)\label{eq:case2upperbound3},
\end{align}
where~\eqref{eq:case2upperbound1} follows from Lemma~\ref{lem:edgecone}  
and~\eqref{eq:case2upperbound3} follows from Lemma~\ref{lem:upexpcvg}. 

Next, we establish item~2 of the proposition.  
For ease of presentation, we assume for now that $S$ does not have a self-loop in order to apply the Blow-up Lemma (Lemma~\ref{lem:blowup}). We will describe how to relax this assumption at the end of this subsection.

Let $\delta$ be given in~\eqref{eq:delta}, and  $\epsilon$ be such that the statement of Lemma~\ref{lem:blowup} is satisfied. 
Then,  
\begin{equation*}
   P_n(W) \geq \bP(\mathcal{E}_n, x(G_n)\in \cX),
\end{equation*}
 where $\mathcal{E}_n$ is introduced in~\eqref{eq:superregular} and the inequality follows from Lemmas~\ref{lem:hamiltonicityofKx} and~\ref{lem:blowup}. 
Choose a closed ball $\mathcal{B}_\delta(x^*)$ centered at $x^*$ with radius $\delta >0$ and $\mathcal{B}_\delta(x^*) \subseteq \rint \cX$, 
it follows that
\begin{align}
    &1-P_n(W) \notag\\
    &\le 1 - \bP(\mathcal{E}_n, x(G_n) \in \cX)\notag\\
    & \le 1 -  \bP(\mathcal{E}_n, x(G_n) \in \mathcal{B}_\delta(x^*))\notag\\
    &= 1 - \bP(\mathcal{E}_n \mid x(G_n) \in \mathcal{B}_\delta(x^*))\bP(x(G_n) \in \mathcal{B}_\delta(x^*))\notag\\
            &\le \bP(x(G_n) \notin \mathcal{B}_\delta(x^*)) + \bP(\neg\mathcal{E}_n \mid x(G_n) \in \mathcal{B}_\delta(x^*))\label{eq:case1uppereq1}.
\end{align}
We now establish upper bounds for the two terms of~\eqref{eq:case1uppereq1}. To bound the first term, we have
\begin{align}
    \bP(x(G_n) \notin \mathcal{B}_\delta(x^*))
    & \le 1 - \bP(\|x^*-x(G_n)\|_2 \le \delta) \notag\\
    & \le 2q\exp\!\left(-\frac{2n\delta^2}{q}\right). \label{eq:case1uppereq2}
\end{align}
By Lemma~\ref{lem:upexpcvg}, the inequality~\eqref{eq:case1uppereq2} holds. The second term in equation~\eqref{eq:case1uppereq1} exponentially decays by Lemma~\ref{lem:expconvregular}. Combining the above bounds, we conclude that there exist positive constants $c$ and $\lambda$ such that
$1 - P_n(W) \leq c \, e^{-\lambda n}$. 
\end{proof}

\begin{remark}\normalfont
At the end of this subsection, we briefly discuss how to handle the proof of the second item of Proposition~\ref{prop:case1_2upperbound} for the case where $S$ has a self-loop. One may start by preprocessing each diagonal block: split it into $2$-by-$2$ sub-blocks of equal size and set the value of two diagonal sub-blocks zero, while leaving the value of two off-diagonal sub-blocks unchanged. 
We denote the resulting step-graphon by $W'$. Then the skeleton graph $S'$ associated with $W'$ has no self-loop. Moreover, using similar arguments in~\cite{chen2025hamiltonicity}, we can show that if $W$ satisfies Condition $\star$, for $\star \in \{A, B, B'\}$ and with $S$ having at least $2$ nodes, 
then so does $W'$.  
Since $W' \leq W$ pointwise, $1-P_n(W)\le 1-P_n(W')$. 
Therefore, for proving the second item Proposition~\ref{prop:case1_2upperbound}, it suffices to consider $W'$. \hfill{\qed}
\end{remark}

\subsection{On the upper bounds of root-$n$ convergence}\label{ssec:rootnupperbound}
In this subsection, we show the upper bounds in~\eqref{eq:case3} and ~\eqref{eq:case4}. We first prove~\eqref{eq:case4}, then~\eqref{eq:case3} is obtained by considering a special case of~\eqref{eq:case4}.
\begin{proposition}\label{prop:rootnupperbounds}
The following two items hold:
\begin{enumerate}
\item If $W$ satisfies Condition $A$ and $B'$, but not $B$, then there exist a nonzero $p^*\in (0,1)$ and a positive constant $c$ such that
\begin{equation}\label{eq:case4upper}
|P_n(W) - p^*| \leq \frac{c}{\sqrt{n}}.  
\end{equation}
\item If $W$ satisfies Condition $B'$, but not $A$, then there exist a positive constant $c$ such that 
\begin{equation}\label{eq:case3upper}
P_n(W) \leq \frac{c}{\sqrt{n}}.
\end{equation}
\end{enumerate}
\end{proposition}
Before proceeding to the proof, we present an auxiliary object $ \Omega_n$ here. We start by defining an affine transformation $T_n: \R^q \to \R^q$ as follows:
\begin{equation}
    x\in \R^q \mapsto T_n(x):= \sqrt{n}(x - x^*) + x^*,\label{eq:affinetransformation}
\end{equation}
which has the inverse given by
$$
\omega\in \R^q \mapsto T^{-1}_n(\omega):=\frac{1}{\sqrt{n}}(\omega - x^*) + x^*.
$$
For each $n\in \mathbb{N}$, we define  
\begin{equation}\label{eq:defsetomegan}
\Omega_n:= \left\{ T_n(x) \mid x \in \cX \right \}.\notag
\end{equation} 
Now we state the following lemma~\cite{gao2025h}:
\begin{lemma}\label{lem:residualeq}
    The following two items hold:
    \begin{enumerate}
        \item $\bP(x(G_n) \in \cX) = \bP(\omega(G_n) \in  \Omega_n)$. 
        \item $\lim_{n\to \infty}\bP\left(x(G_n) \in \cX\right) =  \bP\left(\omega^* \in  \Omega^*\right)$.
    \end{enumerate}
\end{lemma}

We also need the following result:

\begin{lemma}\label{lem:Omega_inclusion}
Let $\mathcal{B}_r(x^*)$ be a closed ball centered at $x^*$ with radius $r>0$, and define $\Omega_n^{\mathcal{B}} := T_n(\mathcal{B}_r(x^*) \cap \cX).$ Then, for all sufficiently large $n$,
\[
\Omega^* \cap \mathcal{B}_{r\sqrt{n}}(x^*) \subseteq \Omega_n^{\mathcal{B}}.
\]
\end{lemma}

\begin{proof}
Let $\omega \in \Omega^* \cap \mathcal{B}_{r\sqrt{n}}(x^*)$, and define
\[
x := T_n^{-1}(\omega) = x^* + \frac{\omega - x^*}{\sqrt{n}}.
\]
Since $\|\omega - x^*\|_2 \le r\sqrt{n}$, we have $\|x - x^*\|_2 \le r,$
and hence $x \in \mathcal{B}_r(x^*)$. We now show that $x \in \cX$. For any $H_k \in \mathcal H(x^*)$, since $\omega \in \Omega^*$ and $v_k^\top x^* = 0$, we have
\[
v_k^\top x
=
v_k^\top\left(x^* + \frac{\omega - x^*}{\sqrt{n}}\right)
=
\frac{1}{\sqrt{n}} v_k^\top \omega \ge 0.
\]
On the other hand, if $H_j \notin \mathcal H(x^*)$, then $v_j^\top x^* > 0$, and hence $v_j^\top x > 0$ for all sufficiently large $n$. Therefore, $x \in \cX$. It follows that $x \in \mathcal{B}_r(x^*) \cap \cX$, and thus $\omega = T_n(x) \in \Omega_n^{\mathcal{B}}.$
This proves the desired inclusion.
\end{proof}

With the lemmas above, we now prove Proposition~\ref{prop:rootnupperbounds}. 

\begin{proof}[Proof of Proposition~\ref{prop:rootnupperbounds}]
We first prove item~1 and establish~\eqref{eq:case4upper}. Let $\mathcal{B}_r(x^*)$ be a closed ball centered at $x^*$ such that
\[
\mathcal{B}_r(x^*)\subset \rint \Delta^{q-1}.
\]
Such a ball exists since $x^*\in \rint\Delta^{q-1}$. 
Note that
\begin{equation}\label{eq:rootn_sandwich}
\bP(\mathcal E_n,\, x(G_n)\in \mathcal{B}_r(x^*)\cap \cX)
\le P_n(W)
\le \bP(x(G_n)\in \cX),
\end{equation}
where the first inequality holds because the event
$\mathcal E_n\cap \{x(G_n)\in \mathcal{B}_r(x^*)\cap \cX\}$ is sufficient for $G_n$ to have a cycle cover (Lemmas~\ref{lem:hamiltonicityofKx} and~\ref{lem:blowup}), and the second inequality follows from Lemma~\ref{lem:edgecone}, which states that $x(G_n)\in \cX$ is necessary for $G_n$ to have a cycle cover. We next bound the two sides of~\eqref{eq:rootn_sandwich}. More specifically, we provide an upper bound for $\bP(x(G_n)\in \cX)$ and a lower bound for $\bP(\mathcal E_n,\, x(G_n)\in \mathcal{B}_r(x^*)\cap \cX)$. 

For the upper bound for $\bP(x(G_n)\in \cX)$, we first note that $\Omega_n = T_n(\cX)$ is convex, which holds because $\cX$ is convex and $T_n$ is affine. Then, 
there exists a positive constant $c_2$ such that
\begin{align}
\bP(x(G_n)\in \cX)
&=\bP(\omega(G_n)\in \Omega_n)\notag\\
&\le \bP(\omega^*\in \Omega_n)+\frac{c_2}{\sqrt n}\label{eq:prop3eq1}\\
&\le \bP(\omega^*\in \Omega^*)+\frac{c_2}{\sqrt n}\label{eq:prop3subset}\\
&=p^*+\frac{c_2}{\sqrt n},\label{eq:prop3eq2}
\end{align}
where the first equality follows from item~1 of Lemma~\ref{lem:residualeq},~\eqref{eq:prop3eq1} follows from Lemma~\ref{lem:berryessen},~\eqref{eq:prop3subset} holds since $\Omega_n\subseteq \Omega^*$, and~\eqref{eq:prop3eq2} follows from Lemma~\ref{lem:residualcase}. 

For the lower bound of $\bP(\mathcal E_n,\, x(G_n)\in \mathcal{B}_r(x^*)\cap \cX)$, 
we have that
\begin{multline}\label{eq:ExinBX}
\bP(\mathcal E_n,\, x(G_n)\in \mathcal{B}_r(x^*)\cap \cX) =  \bP(x(G_n)\in \mathcal{B}_r(x^*)\cap \cX)  \\ - \bP(\neg \mathcal E_n,\, x(G_n)\in \mathcal{B}_r(x^*)\cap \cX)
\end{multline}
Let $\Omega_n^{\mathcal{B}}:=T_n(\mathcal{B}_r(x^*)\cap \cX)$. Since $\mathcal{B}_r(x^*)\cap \cX$ is convex, $\Omega_n^{\mathcal{B}}$ is convex. Appealing to Lemma~\ref{lem:berryessen} again with $Q = \Omega_n^{\mathcal{B}}$, we obtain that there exists a positive constant $c'_1$ such that  
\begin{align}
\bP(x(G_n)\in \mathcal{B}_r(x^*)\cap \cX)
&=\bP(\omega(G_n)\in \Omega_n^{\mathcal{B}})\notag\\
&\ge \bP(\omega^*\in \Omega_n^{\mathcal{B}})-\frac{c'_1}{\sqrt n}.\notag
\end{align}
By Lemma~\ref{lem:Omega_inclusion},
\begin{align}
\bP(\omega^*\in \Omega_n^{\mathcal{B}})
&\ge \bP\big(\omega^*\in \Omega^*,\, \|\omega^*-x^*\|_2\le r\sqrt n\big)\notag\\
&\ge p^*-\bP\big(\|\omega^*-x^*\|_2>r\sqrt n\big).\notag
\end{align}
Since $\omega^*$ is Gaussian with~$x^*$ its mean, standard concentration arguments~\cite{vershynin2018high} imply that there exist positive constants $c''_1$ and $\lambda''_1$ such that
\[
\bP\big(\|\omega^*-x^*\|_2>r\sqrt n\big)\le c''_1 e^{-\lambda''_1 n}.
\]
Hence,
\begin{equation}\label{eq:lower_prob_BcapX}
\bP(x(G_n)\in \mathcal{B}_r(x^*)\cap \cX)
\ge p^*-\frac{c'_1}{\sqrt n}-c''_1 e^{-\lambda''_1 n}.
\end{equation}
Next, we have that
\begin{align}
& \bP(\neg \mathcal E_n,\, x(G_n)\in \mathcal{B}_r(x^*)\cap \cX) \notag \\
& \leq \bP(\neg \mathcal E_n,\, x(G_n)\in \mathcal{B}_r(x^*)) \notag \\
& \leq \bP(\neg \mathcal E_n \mid x(G_n) \in \mathcal{B}_r(x^*)\cap \cX) \bP(x(G_n) \in \mathcal{B}_r(x^*)\cap \cX) \notag \\
& \leq  c'''_1e^{-\lambda_1'''n}, \label{eq:upperboundnegExinBX} 
\end{align}
where~\eqref{eq:upperboundnegExinBX} follows from Lemma~\ref{lem:expconvregular} for some positive constants $c'''_1$ and $\lambda'''_1$.  Combining~\eqref{eq:ExinBX},~\eqref{eq:lower_prob_BcapX}, and~\eqref{eq:upperboundnegExinBX}, we obtain that
\begin{equation}\label{eq:lower_bound_rootn}
\bP(\mathcal E_n,\, x(G_n)\in \mathcal{B}_r(x^*)\cap \cX)
 \\ \ge
p^*-\frac{c'_1}{\sqrt n}-c''_1 e^{-\lambda''_1 n}-c'''_1 e^{-\lambda'''_1 n}.\notag
\end{equation}

Therefore, there exists a positive constant $c_1$ such that 
\begin{equation}\label{eq:lower_bound_rootn_simple}
\bP(\mathcal E_n,\, x(G_n)\in \mathcal{B}_r(x^*)\cap \cX)
\ge p^*-\frac{c_1}{\sqrt n}.
\end{equation}

Finally, combining~\eqref{eq:rootn_sandwich}, \eqref{eq:prop3eq2}, and \eqref{eq:lower_bound_rootn_simple}, we obtain that
\[
p^*-\frac{c_1}{\sqrt n}
\le P_n(W)
\le p^*+\frac{c_2}{\sqrt n},
\]
which implies that there exists a positive constant $c$ such that~\eqref{eq:case4upper} holds.

Next, we establish~\eqref{eq:case3upper}. 
By Lemma~\ref{lem:fullrankness}, if $W$ contains no odd cycle, then $\dim {\cX} = q-1$ (the codimension of $\cX$ in $\R^q$ is thus $1$). Since affine transformation $T_n$ defined in~\eqref{eq:affinetransformation} preserves dimension, $\Omega^*$ also has codimension one. Also, $\omega^*$ is supported on $L$ and $\Omega^*$ is contained in a proper affine subspace of $L$, it follows that $\bP(\omega^* \in \Omega^*) = 0$. The claim then follows from the same arguments above for establishing the upper bound $\bP(x(G_n)\in \cX)$. Specifically, we have that
$$P_n(W) \leq \bP(x(G_n)\in \cX) \leq \bP(\omega^*\in \Omega^*)  + \frac{c}{\sqrt{n}}  = \frac{c}{\sqrt{n}}.$$
This completes the proof. 
\end{proof}
\section{Numerical Validation}\label{sec:numericalstudy}
In this section, we validate the four items of Theorem~\ref{thm:convergencerate} through numerical studies. We consider 11 step-graphons $W_\star$ for $\star = a, b, c, d, e, f, g, h, i, j, k$. For each step-graphon, we sampled $100,000$ random graphs $G_n \sim W_\star$ and computed the empirical probability $p_\star(n)$ that $G_n$ have a cycle cover. Precisely,
\[
p_\star(n) := \frac{\text{number of } G_n\sim W_\star \text{ has a cycle cover}}{100,000}.
\]
For simplicity, we set the graphon value to be 1 on the support of each $W_\star$. All linear fits reported in the following experiments are computed using least-squares regression.
\subsection{Validation of items~1 and~2}
\begin{figure}[t]
\centering
\begin{subfigure}{0.45\textwidth}
        \centering
        \begin{tikzpicture}[scale=2.3]

        \filldraw [fill=black!47, draw=black!47] (0,1) rectangle (0.3, 0.7); 

        \filldraw [fill=black!0, draw=black!0] (0.3, 0.7) rectangle (0.8, 0.2);

        \filldraw [fill=black!0, draw=black!0] (0.8, 0.2) rectangle (1, 0);
        
        \filldraw [fill=black!60, draw=black!60] (0, 0.7) rectangle (0.3, 0.2);

        \filldraw [fill=black!60, draw=black!60] (0.3, 1) rectangle (0.8, 0.7);

        \filldraw [fill=black!28, draw=black!28] (0, 0.2) rectangle (0.3, 0);

        \filldraw [fill=black!28, draw=black!28] (0.8, 1) rectangle (1, 0.7);

        \filldraw [fill=black!50, draw=black!50] (0.3, 0.2) rectangle (0.8, 0);

        \filldraw [fill=black!50, draw=black!50] (0.8, 0.7) rectangle (1, 0.2);

        \draw [draw=black, very thick] (0,0) rectangle (1,1);

        \draw[->, very thick] (0, 1) -- (0, -0.1) node [left] {$s$};
        \draw[->, very thick] (0, 1) -- (1.1, 1) node [above] {$t$};

        \node [above left] at (0, 1) {\small $0$}; 
        \node [left] at (0, 0.7) {\small $\sigma_1$};  
        \node [left] at (0, 0.2) {\small $\sigma_2$};  

        \node [above] at (0.3, 1) {\small $\sigma_1$};
        \node [above] at (0.8, 1) {\small $\sigma_2$};

    \end{tikzpicture}
    \caption{}
    \label{sfig1:graphon}
     \end{subfigure}
\hfill
\begin{subfigure}{0.45\textwidth}
\centering
\begin{tikzpicture}[scale=1.3]

\node [circle,fill=black,inner sep=0.8pt] at (0,0) {};
\draw[->] (0,0) -- (1.55,0) node [above] {\footnotesize $x_2$};
\draw[->] (0,0) -- (0,1.55) node [right] {\footnotesize $x_3$};
\draw[->] (0,0) -- (-0.9,-0.9) node [left] {\footnotesize $x_1$};

\coordinate (A) at (-0.66,-0.66);
\coordinate (B) at (1.2,0);
\coordinate (C) at (0,1.2);

\coordinate (AB) at (0.27,-0.33);
\coordinate (AC) at (-0.33,0.27);
\coordinate (BC) at (0.6,0.6);

\filldraw[fill=gray!0] (A)--(B)--(C)--cycle;
\filldraw[fill=blue!20] (A)--(AB)--(BC)--(AC)--cycle;

\node at (-0.8,-0.2) {\small $\overline \cX$};

\node[circle,fill=blue,inner sep=0.8pt,label=left:{\scriptsize$a$}] at (-0.05,-0.05) {};
\node[circle,fill=blue,inner sep=0.8pt,label=left:{\scriptsize$b$}] at (0.25,0.24) {};
\node[circle,fill=blue,inner sep=0.8pt,label=left:{\scriptsize$c$}] at (0.41,0.40) {};

\node[circle,fill=red,inner sep=0.8pt,label={[yshift=-0.4cm]\scriptsize$d$}] at (0.72,0.09) {};
\node[circle,fill=red,inner sep=0.8pt,label={[yshift=-0.4cm]\scriptsize$e$}] at (0.61,0.145) {};
\node[circle,fill=red,inner sep=0.8pt,label={[yshift=-0.5cm]\scriptsize$f$}] at (0.51,0.18) {};

\end{tikzpicture}
\caption{}
\label{sfig1:edgecone}
\end{subfigure}

\vspace{0.35cm}

\begin{subfigure}{0.45\textwidth}
\centering
\begin{tikzpicture}[scale=1.0]
\tikzset{every loop/.style={}}
\node at (0.2,1.0) {$S$};

\node [circle,fill=black,inner sep=1.2pt,label=below:{$u_1$}] (u1) at (0, 0) {};
\node [circle,fill=black,inner sep=1.2pt,label=below:{$u_2$}] (u2) at (2, 0) {};
\node [circle,fill=black,inner sep=1.2pt,label=above:{$u_3$}] (u3) at (1, 1.5) {};

\path[draw,thick,shorten >=2pt,shorten <=2pt]
(u1) edge[loop left] (u1)
(u1) edge (u2)
(u1) edge (u3)
(u2) edge (u3);

\end{tikzpicture}
\caption{}
\label{sfig1:skeletongraph}
\end{subfigure}
\hfill
\begin{subfigure}{0.45\textwidth}
\centering
\begin{tikzpicture}

\node[align=left] at (0,0) {\small
\begin{tabular}{l}
$x^*_a = (1/2,\, 1/4,\, 1/4)$ \\
$x^*_b = (1/4,\, 3/8,\, 3/8)$ \\
$x^*_c = (1/8,\, 7/16,\, 7/16)$ \\
$x^*_d = (1/8,\, 3/4,\, 1/8)$ \\
$x^*_e = (3/16,\, 5/8,\, 3/16)$ \\
$x^*_f = (7/32,\, 18/32,\, 7/32)$
\end{tabular}
};

\end{tikzpicture}
\caption{}
\label{sfig1:coordinates}
\end{subfigure}

\caption{An example step-graphon $W_\star$ is shown in (a). The associated common skeleton graph is shown in (c), for $\star = a,b,c,d,e,f$. The shaded region in (b) represents $\overline \cX := \cX\cap \Delta^{2}$. The blue points $a$, $b$, and $c$ belong to $\rint \overline \cX$ and correspond to the concentration vectors $x_a^*$, $x_b^*$, and $x_c^*$, respectively. The three red points $d$, $e$, and $f$ belong to $\Delta^2 - \overline \cX$ and correspond to the concentration vectors $x_d^*$, $x_e^*$, and $x_f^*$, respectively. The explicit coordinates of these vectors are given in (d).}
\label{fig:case1&2}

\end{figure}

For the six graphons $W_\star$, where $\star = a, b, c, d, e, f$, a representative example is shown in Fig.~\ref{sfig1:graphon}. These graphons share the same skeleton graph $S$ displayed in Fig.~\ref{sfig1:skeletongraph}, but differ in the partitions $\sigma =(0, \sigma_1, \sigma_2, 1)$, which lead to different concentration vectors $x_\star^*$. The corresponding edge cone is shown in Fig.~\ref{sfig1:edgecone} with the coordinates of the associated six concentration vectors $x^*_\star$ presented in Fig.~\ref{sfig1:coordinates}.

\subsubsection{Validation of item 1}
We first consider the case $\star = a, b, c$, for which $$x^*_\star = \frac{1}{2}(2\alpha_\star, 1-\alpha_\star, 1-\alpha_\star),\quad \alpha_\star \in (0, 1).$$ We set $\alpha_a = \frac{1}{2}$, $\alpha_b = \frac{1}{4}$, $\alpha_c = \frac{1}{8}$. For all three choices, we have $x^*_\star \in  {\rint \cX}$. Hence, $W_\star$ satisfies Condition $B$, by Theorem~\ref{thm:convergencerate}, $P_n(W_\star)$ converges to 1 exponentially fast.

To verify the result, Fig.~\ref{Fig:insidecase} plots the $\log(1-p_\star(n))$ versus $n$ for $\star = a, b, c$. In all cases, the data are approximately linear, indicating that $1-p_\star(n)$ decays at an exponential rate. Furthermore, the fitted decay rates differ substantially among the cases: Case~(a) yields a faster rate of approximately $ e^{-0.163n}$, Case~(b) decays at about $e^{-0.036n}$, and Case~(c) exhibits a slower decay of $e^{-0.010n}$. This difference is explained by the location of \(x_\star^*\) inside the edge cone. More precisely, \(x_a^*\) lies well inside \(\cX\), \(x_c^*\) is much closer to \(\partial \cX\), and \(x_b^*\) lies between them. This is consistent with the fact that the convergence speed of the $H$-property is essentially determined by the event $x(G_n)\in \cX$.

\begin{figure}[htbp]
\centering
\begin{tikzpicture}[scale=0.8]
\begin{axis}[
    width=10cm,
    height=7cm,
    xlabel={Number of nodes $n$},
    ylabel={$\log(1-p_\star(n))$},
    xmin=5, xmax=210,
    ymin=-12, ymax=-0.2,
    grid=both,
    legend style={
    at={(0.47,0.25)},
    anchor=center,
    draw=black,
    fill=white,
    fill opacity=0.9,
    text opacity=1
},
]

\addplot[
    only marks,
    mark=*,
    mark size=2pt,
    blue
] coordinates {
    (10,-3.220876)
    (20,-4.849834)
    (30,-6.388961)
    (40,-7.929407)
    (50,-9.567015)
    (60,-11.512925)
};
\addlegendentry{$\log(1-p_a(n))$}

\addplot[
    domain=10:60,
    samples=100,
    thick,
    blue,
    forget plot
] {-0.163307*x - 1.528927};

\addplot[
    only marks,
    mark=*,
    mark size=2pt,
    orange
] coordinates {
    (10,-1.362034)
    (20,-1.774935)
    (30,-2.181194)
    (40,-2.583408)
    (50,-2.977109)
    (60,-3.362908)
    (70,-3.697686)
    (80,-4.093379)
    (90,-4.489803)
    (100,-4.855761)
    (110,-5.195181)
    (120,-5.521461)
    (130,-5.878136)
    (140,-6.377127)
    (150,-6.676643)
    (160,-6.897805)
    (170,-7.264430)
    (180,-7.775255)
    (190,-7.957577)
    (200,-8.111728)
};
\addlegendentry{$\log(1-p_b(n))$}

\addplot[
    domain=10:200,
    samples=100,
    thick,
    orange,
    forget plot
] {-0.036254*x - 1.144975};

\addplot[
    only marks,
    mark=*,
    mark size=2pt,
    green!70!black
] coordinates {
    (10,-0.749655)
    (20,-0.847616)
    (30,-0.962576)
    (40,-1.079235)
    (50,-1.199691)
    (60,-1.313999)
    (70,-1.425190)
    (80,-1.538098)
    (90,-1.645265)
    (100,-1.736668)
    (110,-1.845127)
    (120,-1.956779)
    (130,-2.057285)
    (140,-2.160313)
    (150,-2.245114)
    (160,-2.345656)
    (170,-2.441887)
    (180,-2.540614)
    (190,-2.641103)
    (200,-2.732303)
};
\addlegendentry{$\log(1-p_c(n))$}

\addplot[
    domain=10:200,
    samples=100,
    thick,
    green!70!black,
    forget plot
] {-0.010475*x - 0.673261};
\end{axis}
\end{tikzpicture}
\caption{Linear fit of $\log(1-p_\star(n))$ for $\star = a, b, c$.}
\label{Fig:insidecase}
\end{figure}
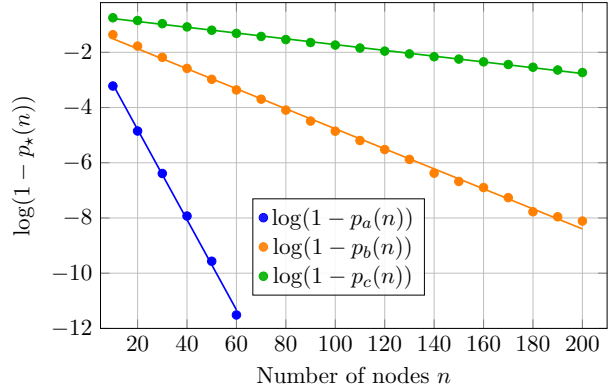

\begin{figure}[htbp]
\centering
\begin{tikzpicture}[scale=0.8]
\begin{axis}[
    width=10cm,
    height=7cm,
    xlabel={Number of nodes $n$},
    ylabel={$\log(p_\star(n))$},
    xmin=5, xmax=210,
    ymin=-12.5, ymax=-0.5,
    grid=both,
    legend style={
    at={(0.5,0.25)},
    anchor=center,
    draw=black,
    fill=white,
    fill opacity=0.9,
    text opacity=1
},
]
\addplot[
    only marks,
    mark=*,
    mark size=2pt,
    blue
] coordinates {
    (10,-2.551046)
    (20,-4.287679)
    (30,-5.878136)
    (40,-7.372437)
    (50,-9.433484)
    (60,-10.414313)
    (70,-11.512925)
};
\addlegendentry{$\log(p_d(n))$}

\addplot[
    domain=10:70,
    samples=100,
    thick,
    blue,
    forget plot
] {-0.152470*x - 1.250932};

\addplot[
    only marks,
    mark=*,
    mark size=2pt,
    orange
] coordinates {
    (10,-1.210338)
    (20,-1.740466)
    (30,-2.201958)
    (40,-2.628050)
    (50,-3.058893)
    (60,-3.421740)
    (70,-3.809045)
    (80,-4.184557)
    (90,-4.571245)
    (100,-4.878815)
    (110,-5.295613)
    (120,-5.692483)
    (130,-5.962739)
    (140,-6.265901)
    (150,-6.802395)
    (160,-6.949856)
    (170,-7.369199)
    (180,-7.901007)
    (190,-8.145630)
    (200,-8.294050)
};
\addlegendentry{$\log(p_e(n))$}

\addplot[
    domain=10:200,
    samples=100,
    thick,
    orange,
    forget plot
] {-0.037124*x - 1.121094};

\addplot[
    only marks,
    mark=*,
    mark size=2pt,
    green!70!black
] coordinates {
    (10,-0.795206)
    (20,-1.004520)
    (30,-1.201196)
    (40,-1.337690)
    (50,-1.482060)
    (60,-1.615706)
    (70,-1.743587)
    (80,-1.868802)
    (90,-1.973858)
    (100,-2.103734)
    (110,-2.200487)
    (120,-2.317538)
    (130,-2.437030)
    (140,-2.535648)
    (150,-2.617610)
    (160,-2.731213)
    (170,-2.808946)
    (180,-2.932971)
    (190,-3.034897)
    (200,-3.131211)
};
\addlegendentry{$\log(p_f(n))$}

\addplot[
    domain=10:200,
    samples=100,
    thick,
    green!70!black,
    forget plot
] {-0.011753*x - 0.859716};

\end{axis}
\end{tikzpicture}
\caption{Linear fit of $\log(p_\star(n))$ for $\star = d, e, f$.}
\label{Fig:outsidecase}
\end{figure}
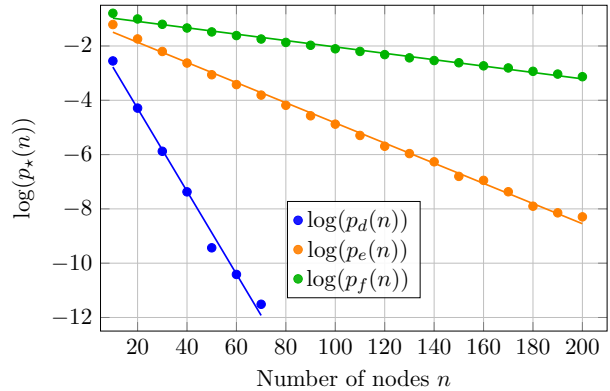

\subsubsection{Validation of item~2 with Condition $A$ satisfied}
We next consider the case $\star = d, e, f$, where $$x^*_\star = \frac{1}{2}(1-\beta_\star, 2\beta_\star, 1-\beta_\star),\quad \beta_\star \in(0, 1).$$ We set $\beta_d = \frac{3}{4}$, $\beta_e = \frac{5}{8}$, $\beta_f = \frac{18}{32}$. For these three graphons, $x^*_\star \notin  {\cX}$, so Condition $B'$ fails, and hence $P_n(W_\star)\to 0$. In addition, Fig.~\ref{sfig1:skeletongraph} contains the odd cycle $(u_1, u_1)$, so Condition $A$ satisfied.

Fig.~\ref{Fig:outsidecase} plots $\log(p_\star(n))$ with respect to $n$ for $\star = d, e, f$. Again, the curves are approximately linear, indicating the exponential decay. The fitted rates are approximately $e^{-0.152n}$, $e^{-0.037n}$, $e^{-0.012n}$ for Case~(d), (e) and (f), respectively. Among the three cases$,$ \(x_d^*\) is furthest from \(\cX\)$,$ while \(x_f^*\) is closest to it. Accordingly$,$ the decay becomes slower as \(x_\star^*\) approaches the boundary of the edge cone. This again supports the conclusion that the geometry of \(\cX\) plays the central role in determining the convergence rate.

\subsubsection{Validation of item~2 with Condition $A$ not satisfied}
We now verify the remaining part of item~2 using three additional step-graphons $W_\star$, for $\star = g, h, i$. An example is shown in Fig.~\ref{sfig4:Wg}. Their concentration vectors are of the form: $$x^*_\star = (\sigma_\star, 1-\sigma_\star), \quad\sigma_\star \in (0, 1).$$ We set $\sigma_g = \frac{7}{16}$, $\sigma_h = \frac{3}{8}$, $\sigma_i = \frac{1}{4}$, respectively. In this case, $\cX = \{(1/2,\,1/2)\}$, and 
$x_\star \notin \cX$ for all three choices. The associated skeleton graph $S$ is given as $S = (V(S),E(S))$, where $V(S) = \{u_1,u_2\}$ and $E(S) = \{(u_1,u_2)\}$ which contain no odd cycle, while still $P_n(W_\star) \to 0$.

Figure~\ref{fig:ghi_exponential_fit} plots $\log(p_\star(n))$ versus $n$ for $\star=g,h,i$. The data again follow an approximately linear trend, confirming exponential decay. The fitted rates are about $e^{-0.014n}$, $e^{-0.039n}$, and $e^{-0.173n}$ for Cases~(g), (h), and (i), respectively. The difference in rates is explained by the relative positions of $x^*$ with respect to $\cX$.

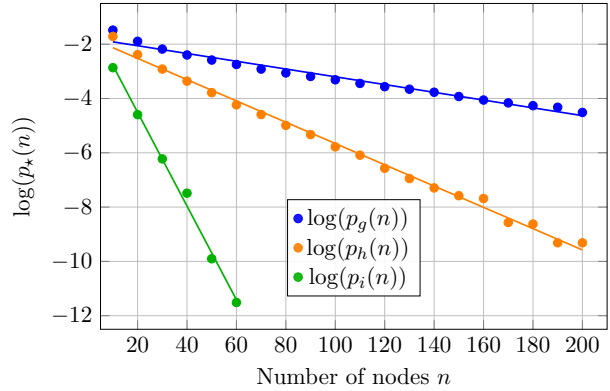
\begin{figure}[htbp]
\centering
\begin{tikzpicture}[scale=0.8]
\begin{axis}[
    width=10cm,
    height=7cm,
    xlabel={Number of nodes $n$},
    ylabel={$\log(p_\star(n))$},
    xmin=5, xmax=210,
    ymin=-12.5, ymax=-0.5,
    grid=both,
    legend style={
        at={(0.5,0.25)},
        anchor=center,
        draw=black,
        fill=white,
        fill opacity=0.9,
        text opacity=1
    },
]

\addplot[
    only marks,
    mark=*,
    mark size=2pt,
    blue
] coordinates {
    (10,-1.482233)
    (20,-1.894257)
    (30,-2.179129)
    (40,-2.397116)
    (50,-2.583888)
    (60,-2.746687)
    (70,-2.916736)
    (80,-3.058033)
    (90,-3.189074)
    (100,-3.311265)
    (110,-3.445776)
    (120,-3.564893)
    (130,-3.657768)
    (140,-3.769656)
    (150,-3.926629)
    (160,-4.057049)
    (170,-4.162409)
    (180,-4.264421)
    (190,-4.323003)
    (200,-4.514416)
};
\addlegendentry{$\log(p_g(n))$}

\addplot[
    domain=10:200,
    samples=100,
    thick,
    blue,
    forget plot
] {-0.014300*x - 1.770772};

\addplot[
    only marks,
    mark=*,
    mark size=2pt,
    orange
] coordinates {
    (10,-1.712246)
    (20,-2.385206)
    (30,-2.915813)
    (40,-3.361304)
    (50,-3.784950)
    (60,-4.233607)
    (70,-4.590282)
    (80,-4.990833)
    (90,-5.328777)
    (100,-5.782826)
    (110,-6.087975)
    (120,-6.571283)
    (130,-6.948577)
    (140,-7.293418)
    (150,-7.581100)
    (160,-7.684284)
    (170,-8.568486)
    (180,-8.622554)
    (190,-9.315701)
    (200,-9.315701)
};
\addlegendentry{$\log(p_h(n))$}

\addplot[
    domain=10:200,
    samples=100,
    thick,
    orange,
    forget plot
] {-0.039199*x - 1.737888};

\addplot[
    only marks,
    mark=*,
    mark size=2pt,
    green!70!black
] coordinates {
    (10,-2.865230)
    (20,-4.594230)
    (30,-6.224658)
    (40,-7.487574)
    (50,-9.903488)
    (60,-11.512925)
};
\addlegendentry{$\log(p_i(n))$}

\addplot[
    domain=10:60,
    samples=100,
    thick,
    green!70!black,
    forget plot
] {-0.172655*x - 1.055101};

\end{axis}
\end{tikzpicture}
\caption{Linear fit of $\log(p_\star(n))$ for $\star = g, h, i$.}
\label{fig:ghi_exponential_fit}
\end{figure}

\subsection{Validation of items~3 and~4}
In this subsection, we illustrate the $n^{-1/2}$ convergence rate by graphons $W_j$ and $W_k$ shown in Fig.~\ref{fig:graphongh}.

\subsubsection{Validation of item~3}
We first consider the graphon $W_j$ shown in Fig.~\ref{sfig4:Wg}. Its concentration vector $x_j^* = (1/2,\,1/2) \in \cX = \{(1/2,\,1/2)\}$, i.e, $\sigma_j = \frac{1}{2}$.
Therefore, $W_j$ satisfies Condition $B'$ but not Condition $A$. By Theorem~\ref{thm:convergencerate}, $P_n(W_j) \to 0$ as $n \to \infty$, and decay rate should be of order $n^{-1/2}$.

In Fig.~\ref{Fig:loglog_pj_fit}, we plot $\log p_j(n)$ versus $\log n$ and perform a linear fit. The fitted slope is approximately $-0.510$ which is close to $-1/2$, indicating the convergence rate $n^{-1/2}$ of $P_n(W_j)$.

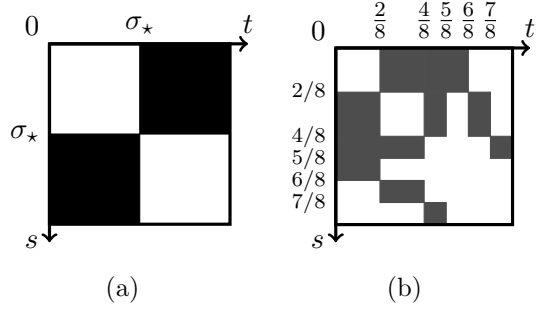
\begin{figure}[t]
    \centering
    \subfloat[\label{sfig4:Wg}]{
\begin{tikzpicture}[scale=2.4]

    \filldraw[fill=black!0, draw=black!0] (0,1) rectangle (0.5,0.5);

    \filldraw[fill=black!0, draw=black!0] (0.5,0.5) rectangle (1,0);

    \filldraw[fill=black, draw=black] (0.5,1) rectangle (1,0.5);

    \filldraw[fill=black, draw=black] (0,0.5) rectangle (0.5,0);

    \draw[draw=black, very thick] (0,0) rectangle (1,1);
  
    \draw[->, very thick] (0,1) -- (0,-0.1) node[left] {$s$};
    \draw[->, very thick] (0,1) -- (1.1,1) node[above] {$t$};

    \node[left] at (0,0.5) {$\sigma_\star$};
    \node[above] at (0.5,1) {$\sigma_\star$};

    \node[left] at (0, 1.1) {$0$};
\end{tikzpicture}
}
\centering
\subfloat[\label{sfig4:Wh}]{
   \begin{tikzpicture}[scale=2.35]

    \filldraw [fill=black!0, draw=black!0] (0, 0.25) rectangle (0.25, 0.125); 
    \filldraw [fill=black!0, draw=black!0] (0.75, 1) rectangle (0.875, 0.75); 

    \filldraw [fill=black!0, draw=black!0] (0, 0.125) rectangle (0.25, 0); 
    \filldraw [fill=black!0, draw=black!0] (0.875, 1) rectangle (1, 0.75); 

    \filldraw [fill=black!0, draw=black!0] (0.25, 0.375) rectangle (0.5, 0.25); 
    \filldraw [fill=black!0, draw=black!0] (0.625, 0.75) rectangle (0.75, 0.5); 

    \filldraw [fill=black!0, draw=black!0] (0.25, 0.125) rectangle (0.5, 0); 
    \filldraw [fill=black!0, draw=black!0] (0.875, 0.75) rectangle (1, 0.5); 

    \filldraw [fill=black!0, draw=black!0] (0.5, 0.375) rectangle (0.625, 0.25); 
    \filldraw [fill=black!0, draw=black!0] (0.625, 0.5) rectangle (0.75, 0.375); 

    \filldraw [fill=black!0, draw=black!0] (0.5, 0.25) rectangle (0.625, 0.125); 
    \filldraw [fill=black!0, draw=black!0] (0.75, 0.5) rectangle (0.875, 0.375); 

    \filldraw [fill=black!0, draw=black!0] (0.75, 0.375) rectangle (0.875, 0.25); 
    \filldraw [fill=black!0, draw=black!0] (0.625, 0.25) rectangle (0.75, 0.125); 

    \filldraw [fill=black!0, draw=black!0] (0.75, 0.125) rectangle (0.875, 0); 
    \filldraw [fill=black!0, draw=black!0] (0.875, 0.25) rectangle (1, 0.125); 

    \filldraw [fill=black!0, draw=black!0] (0,1) rectangle (0.25, 0.75); 

    \filldraw [fill=black!0, draw=black!0] (0.25, 0.75) rectangle (0.5, 0.5);

    \filldraw [fill=black!0, draw=black!0] (0.5, 0.5) rectangle (0.625, 0.375);

    \filldraw [fill=black!0, draw=black!0] (0.625, 0.375) rectangle (0.75, 0.25);

    \filldraw [fill=black!0, draw=black!0] (0.75, 0.25) rectangle (0.875, 0.125);

    \filldraw [fill=black!0, draw=black!0] (0.875, 0.125) rectangle (1, 0);


    \filldraw [fill=black!70, draw=black!70] (0, 0.75) rectangle (0.25, 0.5); 
    \filldraw [fill=black!70, draw=black!70] (0.25, 1) rectangle (0.5, 0.75); 

    \filldraw [fill=black!70, draw=black!70] (0, 0.5) rectangle (0.25, 0.375); 
    \filldraw [fill=black!70, draw=black!70] (0.5, 1) rectangle (0.625, 0.75); 

    \filldraw [fill=black!70, draw=black!70] (0, 0.375) rectangle (0.25, 0.25); 
    \filldraw [fill=black!70, draw=black!70] (0.625, 1) rectangle (0.75, 0.75); 

    \filldraw [fill=black!70, draw=black!70] (0.25, 0.5) rectangle (0.5, 0.375); 
    \filldraw [fill=black!70, draw=black!70] (0.5, 0.75) rectangle (0.625, 0.5); 

    \filldraw [fill=black!70, draw=black!70] (0.25, 0.25) rectangle (0.5, 0.125); 
    \filldraw [fill=black!70, draw=black!70] (0.75, 0.75) rectangle (0.875, 0.5); 

    \filldraw [fill=black!70, draw=black!70] (0.5, 0.125) rectangle (0.625, 0); 
    \filldraw [fill=black!70, draw=black!70] (0.875, 0.5) rectangle (1, 0.375); 
        \draw [draw=black, very thick] (0,0) rectangle (1,1);

        \draw[->, very thick] (0, 1) -- (0, -0.1) node [left] {$s$};
        \draw[->, very thick] (0, 1) -- (1.1, 1) node [above] {$t$};

        \node [left] at (0, 0.75) {\scriptsize$2/8$};
        \node [left] at (0, 0.5) {\scriptsize$4/8$};
        \node [left] at (0, 0.375) {\scriptsize$5/8$};
        \node [left] at (0, 0.25) {\scriptsize$6/8$};
        \node [left] at (0, 0.125) {\scriptsize$7/8$};

        \node [above] at (0.25, 1) {$\frac{2}{8}$};
        \node [above] at (0.5, 1) {$\frac{4}{8}$};
        \node [above] at (0.625, 1) {$\frac{5}{8}$};
        \node [above] at (0.75, 1) {$\frac{6}{8}$};
        \node [above] at (0.875, 1) {$\frac{7}{8}$};

        \node [left] at (0, 1.1) {$0$}; 

    \end{tikzpicture}
}
\caption{(a) A representative illustration of the step-graphon $W_\star$ for $\star \in \{g,h,i,j\}$; the visualization is generated with $\sigma_\star = \frac{1}{2}$ for simplicity. (b) The step-graphon $W_k$.}
\label{fig:graphongh}
\end{figure}

\begin{figure}[htbp]
\centering
\begin{tikzpicture}[scale=0.8]
\begin{axis}[
    width=10cm,
    height=6.2cm,
    xlabel={$\log n$},
    ylabel={$\log p_j(n)$},
    xmin=4.5, xmax=7.15,
    ymin=-3.9, ymax=-2.4,
    grid=both,
    legend style={
        at={(0.98,0.95)},
        anchor=north east,
        draw=gray,
        fill=white
    },
]

\addplot[
    blue,
    only marks,
    mark=*,
    mark size=2pt
] coordinates {
    (4.60517019,-2.5098553)
    (5.29831737,-2.87316072)
    (5.70378247,-3.09887303)
    (5.99146455,-3.22087783)
    (6.21460810,-3.33092533)
    (6.39692966,-3.42713076)
    (6.55108034,-3.50655790)
    (6.68461173,-3.55225286)
    (6.80239476,-3.66282192)
    (6.90775528,-3.69449519)
    (7.00306546,-3.75075486)
    (7.09007684,-3.76878884)
};
\addlegendentry{$\log(p_j(n))$}

\addplot[
    red,
    thick,
    domain=4.60517019:7.09007684,
    samples=200
]
{-0.509567978682325*x - 0.17099024004626778};
\addlegendentry{Fit: $-0.510\,\log n - 0.171$}

\end{axis}
\end{tikzpicture}
\caption{Plot of $\log(p_j(n))$ versus $\log n$ with its linear fit.}
\label{Fig:loglog_pj_fit}
\end{figure}
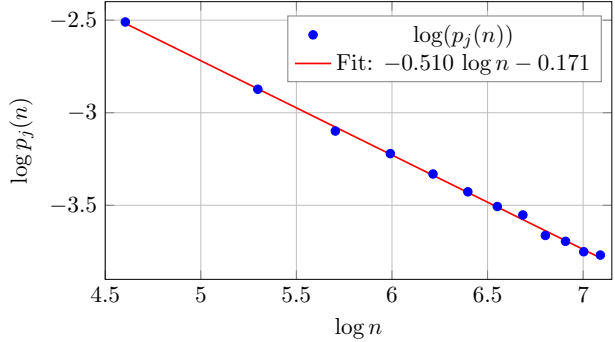

\subsubsection{Validation of item~4}

Finally, we consider the step-graphon $W_k$ shown in Fig.~\ref{sfig4:Wh}, it can be verified that the concentration vector $x^*_k \in \partial\cX$, so $P_n(W_k) \to p^*$, with $p^* = 0.1669$ (See Example~2 in \cite{gao2025h}). To examine the convergence rate, we plot $\log |p_k(n)-p^*|$ versus $\log n$ and perform a linear fit in Fig.~\ref{Fig:loglog_pk_boundary}. The fitted line is
$\log |p_k(n)-p^*| = -0.510 \log n - 0.883,$
which indicates that $|p_k(n)-p^*|$ decays at the rate $n^{-1/2}$.

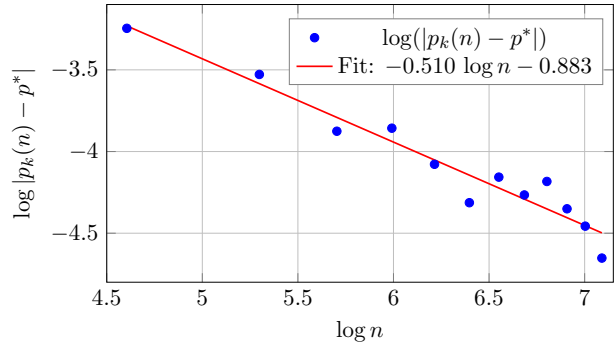
\begin{figure}[htbp]
\centering
\begin{tikzpicture}[scale=0.8]
\begin{axis}[
    width=10cm,
    height=6.2cm,
    xlabel={$\log n$},
    ylabel={$\log |p_k(n)-p^*|$},
    xmin=4.5, xmax=7.15,
    ymin=-4.8, ymax=-3.1,
    grid=both,
    legend style={
        at={(0.98,0.95)},
        anchor=north east,
        draw=gray,
        fill=white
    },
]

\addplot[
    blue,
    only marks,
    mark=*,
    mark size=2pt
] coordinates {
    (4.60517019,-3.24573327)
    (5.29831737,-3.52812209)
    (5.70378247,-3.87569112)
    (5.99146455,-3.85706141)
    (6.21460810,-4.07807754)
    (6.39692966,-4.31324707)
    (6.55108034,-4.15664555)
    (6.68461173,-4.26584482)
    (6.80239476,-4.18317583)
    (6.90775528,-4.35052801)
    (7.00306546,-4.45675022)
    (7.09007684,-4.65226172)
};
\addlegendentry{$\log(|p_k(n)-p^*|)$}

\addplot[
    red,
    thick,
    domain=4.60517019:7.09007684,
    samples=200
]
{-0.5098822801542713*x - 0.882906337786286};
\addlegendentry{Fit: $-0.510\,\log n - 0.883$}

\end{axis}
\end{tikzpicture}
\caption{Plot of $\log(|p_k(n)-p^*|)$ versus $\log(n)$ with its linear fit, where $p^* = 0.1669$.}
\label{Fig:loglog_pk_boundary}
\end{figure}
Taken together, the experiments in this section clearly illustrate the two regimes stated by Theorem~\ref{thm:convergencerate}: exponential convergence for item~1 and item~2, and the $n^{-1/2}$ order convergence in item~3 and item~4. They also show that in the exponential regime, the rate is strongly influenced by the location of $x^*$ relative to $\cX$.

\section{Conclusions}\label{sec:conclusions}

In this paper, we have investigated the convergence rate of the probability $P_n(W)$ that $G_n\sim W$ has a cycle cover. 
We have demonstrated that there are two different types of convergence rates, namely, exponential convergence and root-$n$ convergence, 
which depend on the geometric relation between the concentration vector $x^*$ and the edge cone $\cX$. 
Specifically, if there exists an open neighborhood $U$ of $x^*$ in $\Delta^{q-1}$ such that either $U\subseteq \rint \cX$ or $U\subseteq \Delta^{q-1} - \cX$, then $P_n(W)$ converges to its limit value exponentially fast. Otherwise, the convergence rate is $n^{-1/2}$.
We have also validated the results through numerical simulations.

\printbibliography
\end{document}